\documentclass[12pt]{iopart}

\usepackage{graphicx,epsfig,latexsym,enumerate,iopams,setstack}
\usepackage{caption}
\usepackage{subfig}
\usepackage{overpic}
\usepackage{vector}
\usepackage{hyperref}

\newtheorem{theorem}{Theorem}[section]
\newtheorem{proposition}{Proposition}[section]
\newtheorem{corollary}{Corollary}[section]
\newtheorem{remark}[theorem]{Remark}

\newtheorem{assumption}[theorem]{Assumption}

\newtheorem{definition}[theorem]{Definition}
\newenvironment{proof}[1][Proof]{\begin{trivlist}
\item[\hskip \labelsep {\bfseries #1}]}{\end{trivlist}}
\eqnobysec

\newcommand{\dd}{\mathrm{d}}

\newcommand{\cvd}{\hfill$\square$}

\newcommand{\ignore}[1]{}

\newcommand{\notmid}{\mid\kern-0.5em\not\kern0.5em}

\newcommand{\ba}{\begin{array}}
\newcommand{\ea}{\end{array}}

\begin{document}
\title{Analysis of the Magneto-acoustic Tomography with Magnetic Induction (MAT-MI)}

\author{Lingyun Qiu$^1$ and Fadil Santosa$^1$}
\address{$^1$ Institute for Mathematics and its Applications, University of Minnesota, Minneapolis, MN 55455, USA.}
\ead{\mailto{qiu.lingyun@ima.umn.edu}, \mailto{santosa@umn.edu}}



\begin{abstract}
Magnetoacoustic tomography with magnetic induction (MAT-MI) is a
coupled-physics medical imaging modality for determining conductivity
distribution in biological tissue.  The capability of MAT-MI to
provide high resolution images has been demonstrated
experimentally. MAT-MI involves two steps. The first step is a
well-posed inverse source problem for acoustic wave equation, which
has been well studied in the literature. This paper concerns
mathematical analysis of the second step, a quantitative
reconstruction of the conductivity from knowledge of the internal
data recovered in the first step, using techniques such as time reversal.
The problem is modeled by a system derived from Maxwell's
equations.  We show that a single internal data determines the
conductivity.  A global Lipschitz type stability estimate is obtained.
A numerical approach for recovering the conductivity is proposed and
results from computational experiments are presented.
\end{abstract}

\maketitle



\section{Introduction}
\label{sec:Intro}
Electrical conductivity of the biological tissues can provide
important information for clinical and research purposes. Conductivity
imaging has been a subject of research for decades and the literature
is vast.

Magnetoacoustic tomography with magnetic induction (MAT-MI) is a new
noninvasive modality for imaging electrical conductivity distribution
of biological tissue \cite{Xu2005,Li2006,Mariappan2014}.  In the
experiments, the biological tissue is placed in a static magnetic
field. A pulsed magnetic field is applied and induces an eddy
current inside the conductive tissue.
Consequently, the Lorentz force, the force acting on
currents in the static magnetic field, causes vibrations and the tissue emits
ultrasound waves. The ultrasonic signals are measured around the
object.  MAT-MI belongs to the class of coupled-physics imaging method which
is often refered to as `hybrid imaging'.
For a review on hybrid imaging methods that recover
electrical conductivity distribution, we refer to \cite{Widlak2012}.

Hybrid imaging typically involves two inverse
problems. In MAT-MI the two steps are decoupled. The first step involves an
inverse source problem for the acoustic wave equation. This problem
has been studied
extensively in many works including \cite{Finch2009, Haltmeier2005, Hristova2008, Kuchment2008, Stefanov2009, Qian2011}. The second step, the focus of this work,
is to reconstruct the
spatially varying electrical conductivity from knowledge of the acoustic source.


In the MAT-MI experiment, the object to be imaged is placed in a constant static
magnetic background field $\bvec{B}_0=(0,0,1)$.  A pulsed magnetic stimulation is
introduced.  The pulsed field is of the form $\bvec{B}_1 u(t)$,
where the vector field $\bvec{B}_1$ is a constant and $u(t)$ is the time variation.
The magnetic permeability of biological tissue is approximately equal to that of a vacuum.
Therefore the tissue does not have any noticeable effect on the magnetic field itself.
As a result, the time-dependence of the electromagnetic fields is $u(t)$ and we need only
to consider their spatial dependence.  Because the electric field will depend on conductivity
$\sigma$, we write it as $\bvec{E}_\sigma$.  Let $\Omega$ denote the domain to be imaged.
Then it can be shown that the electric field satisfies
\begin{equation}\label{eqn:MAT-MI}
\left\{
\begin{array}{lll}
  \nabla \times \bvec{E}_\sigma & = \bvec{B}_1, &\qquad \mbox{ in } \Omega,
  \\
  \nabla \cdot (\sigma \bvec{E}_\sigma) & =0, &\qquad \mbox{ in } \Omega,
  \\
  \sigma \bvec{E}_\sigma \cdot \nu &= 0 , &\qquad \mbox{ on } \partial\Omega .
  \end{array}
  \right.
\end{equation}
The first step in the MAT-MI inverse problem is to recover the acoustic source
in the scalar wave equation from observed data at a set of locations.  The
acoustic source is related to the electromagnetic field; knowledge of the
acoustic source in this model is equivalent to knowing the quantity
$\nabla \cdot(\sigma \bvec{E}_\sigma \times \bvec{B}_0)$ throughout $\Omega$.

In this paper, we focus on the second step of
MAT-MI, i.e., reconstruction of the conductivity $\sigma$ from the internal
data given by $\nabla \cdot(\sigma \bvec{E}_\sigma \times \bvec{B}_0)$.
Our main result is that, if the conductivity is a priori
known near the boundary, then it can be uniquely and stably
reconstructed from one internal data. More precisely, the main result
of this work reads as follows.

\begin{theorem}\label{thm:nonlinear-stab2}
Denote the forward map, the map from conductivity to acoustic source, as $F(\sigma):=
\nabla \cdot (\sigma \bvec{E}_\sigma \times \bvec{B}_0)$.
Suppose that $\sigma_1$ and $\sigma_2$ satisfy Assumption~\ref{asmp:sigma} and the support of $\sigma_1 -\sigma_2$ is away from the boundary of $\Omega$ at a distance greater than some constant $r_0>0$. Then, there exists a constant $K$, which only depends on $r_0$, $\lambda$, $\Lambda$ and $\Omega$, such that,
if
\begin{equation}\label{sig-slow}
\|\nabla \sigma_1 \|_{L^\infty} < K,
\end{equation}
then the inequality
  \begin{equation}\label{eqn:nonlinear-stab2}
    \|\sigma_1 - \sigma_2\|_{L^2(\Omega)} \leq 4 \|F(\sigma_1) - F(\sigma_2)\|_{L^2(\Omega)},
  \end{equation}
  holds true.
\end{theorem}

During the completion of this work, we discovered a recent paper by Ammari, Boulier
and Millien \cite{Ammari2015}.  Their work also focused on the conductivity reconstruction
aspect of MAT-MI.  What is different is that the authors chose to reconstruct first
the current density in the medium.  They propose methods to solve for conductivity from
current density. In our approach, we directly deal with the relationship between
the acoustic source and the electromagnetic field, and propose a method that finds
the conductivity from the acoustic source.

The rest of the paper is organized as follows. Section 2 introduces
the notation used and basic results needed.  In
Section~\ref{sec:model}, we study the mathematical model of the second
step of MAT-MI
and the linearized version of this problem. Section~\ref{sec:stab} is devoted
to addressing the uniqueness and stability estimate of both linearized
and nonlinear problems.  In Section 5, we propose an numerical method
for solving the inverse problem and present some results from
computational experiments.  A final section discusses our findings.

\section{Notations and preliminaries}
\label{sec:prelim}
We begin by introducing the notations for the the mathematical
analysis. Throughout this paper, the standard notations for continuous
differentiable function spaces and Sobolev spaces are used. Let
$\Omega$ be a bounded domain in $\mathbb{R}^3$ with Lipschitz boundary
$\partial \Omega$. A typical point $x=(x_1,x_2,x_3)\in\mathbb{R}^3$
denotes the spatial variable. We use the notation $C^\infty(\Omega)$
for infinitely differentiable functions on $\Omega$ and
$C_0^\infty(\Omega)$ is a subset of $C^\infty(\Omega)$ which contains
the functions with compact support.  We use $\langle \cdot , \cdot \rangle$
to denote the inner product in the Hilbert space $L^2(\Omega)$.
For $p\geq 1$, we denote by $W^{1,p}(\Omega)$ the $L^p$-based Sobolev
spaces on $\Omega$ with the usual norm,
\[
\|u\|_{W^{1,p}(\Omega)} = \|u\|_{L^p(\Omega)} + \sum_{n=1} ^3 \left\|\frac{\partial u}{\partial x_n} \right\|_{L^p(\Omega)} .
\]
In the case $p=2$, we use the notation $H^1(\Omega) =
W^{1,2}(\Omega)$, which is a Hilbert space.  The Sobolev space
$H^1_0(\Omega)$ is defined as the closure of $C_0^\infty(\Omega)$ in
$H^1(\Omega)$. The dual space of $H^1_0(\Omega)$ is denoted by
$H^{-1}(\Omega)$.  If there is no danger of confusion, we omit the
domain $\Omega$ and abbreviate with $L^2$, $W^{1,p}$, $H^1$, $H^1_0$ and
$H^{-1}$.  In the following, we do not distinguish in the notation for
inner product, function spaces and the corresponding norms between
scalar- and vector-valued functions.

\begin{assumption}\label{asmp:sigma}
   Let $\sigma$ be a positive function belongs to $W^{1,\infty}$ and assume that
\begin{equation}\label{ellipticity}
   \sigma(x) \geq \lambda, \quad \forall x\in \Omega.
  \end{equation}
  and
  \[
   \|\sigma\|_{W^1,\infty} \leq \Lambda
  \]
for some constants $\lambda,\Lambda >0$.
\end{assumption}


We start with stating several useful results on the elliptic partial differential equations with Neumann boundary condition.

\begin{definition}\label{def:solution}
We say that $u\in H^1$ is a weak solution of the Neumann boundary value problem,
  \begin{equation}\label{eqn:N-bd}
    \left\{
     \begin{array}{lll}
   \nabla \cdot (\sigma \nabla u) & = - \nabla \cdot \bvec{E}, &\qquad \mbox{ in } \Omega,
  \\
     (\sigma \nabla u + \bvec{E}) \cdot \nu & =  0 , &\qquad \mbox{ on } \partial\Omega,
  \end{array}
    \right.
  \end{equation}
  if
  \[
  \int_{\Omega} \sigma \nabla u \cdot \nabla \varphi \,\dd x = -\int_{\Omega} \bvec{E} \cdot \nabla \varphi \,\dd x, \quad \forall \varphi \in H^1.
  \]
\end{definition}

\medskip\medskip

We need the following regularity result and standard energy estimate of the gradient.

\begin{proposition}\label{prop:reg}
Suppose that $\sigma$ satisfies Assumption~\ref{asmp:sigma}.
For field $\bvec{E}\in L^2$, the Neumann problem \eref{eqn:N-bd} has a solution $u\in H^1$.
The solution $u$ is unique up to an additive constant and satisfies the estimate,
  \begin{equation}\label{grad-est}
    \|\nabla u\|_{L^2} \leq \lambda^{-1} \|\bvec{E}\|_{L^2}.
  \end{equation}
\end{proposition}

\begin{proof}
The proof of the existence and uniqueness up to an additive constant is a standard result by the Lax-Milgram Theorem. We refer the readers to \cite{Taylor2011}. In the following, we prove the gradient estimate \eref{grad-est}.

It follows from the ellipticity condition \eref{ellipticity} that
  \[
    \lambda \|\nabla u \|_{L^2}^2 \leq \int_{\Omega} \sigma |\nabla u|^2 \, \dd x.
    \]
Taking the test function $\varphi$ in Definition~\ref{def:solution} to be the solution $u$, we have that
\[
\int_{\Omega} \sigma \nabla u \cdot \nabla u \, \dd x =  - \int_{\Omega} \bvec{E} \cdot  \nabla u \, \dd x.
\]
Consequently, applying the Cauchy-Schwarz inequality, we obtain that
  \[
   \lambda \|\nabla u \|_{L^2}^2 \leq \left| - \int_{\Omega} \bvec{E} \cdot  \nabla u \, \dd x \right| \leq \|\nabla u \|_{L^2} \|\bvec{E}\|_{L^2},
  \]
  and \eref{grad-est} follows. \cvd
\end{proof}

\section{Analysis of the forward problem}
\label{sec:model}

\subsection{The forward problem}

The second step of MAT-MI is modeled by \eref{eqn:MAT-MI},
where $\nu$ is the unit outer normal vector of $\partial \Omega$ and
$\bvec{B}_1 = (0,0,1)$ is a constant vector.    The data for this
inverse problem is the acoustic source recovered from the first step, namely,
$\nabla \cdot (\sigma
\bvec{E}_\sigma \times \bvec{B}_0)$ with $\bvec{B}_0 = (0,0,1)$.
The inverse problem of the second step of MAT-MI consists of
reconstruction of conductivity $\sigma$ from knowledge of $\nabla
\cdot (\sigma \bvec{E}_\sigma \times \bvec{B}_0)$.

We refer the readers to
\cite{Colton2013} for the regularity results of the Maxwell's
equations. In Proposition~\ref{prop:E-bd}, we show some regularity
results of our reduced system \eref{eqn:MAT-MI}.

\begin{definition}\label{def:vec-solution}
We say that $\bvec{E}_\sigma\in L^2$ is a weak solution of the \eref{eqn:MAT-MI}
  if
   \[
    \int_{\Omega} \bvec{E}_\sigma \cdot (\nabla \times \bvec{\Phi}) \,\dd x = \int_{\Omega} \bvec{\Phi} \cdot \bvec{B}_1 \,\dd x, \quad \forall \bvec{\Phi} \in H_0^1,
    \]
    and
  \[
  \int_{\Omega} \sigma \bvec{E}_\sigma \cdot \nabla \varphi \,\dd x = 0, \quad \forall \varphi \in H^1.
  \]
\end{definition}

\medskip\medskip

We define the forward problem as
\begin{equation}\label{forward-F}
\begin{array}{rrl}
    F: & W^{1,\infty} \rightarrow  & L^2 ,
    \\
    &\sigma \mapsto & \nabla \cdot (\sigma \bvec{E}_\sigma \times \bvec{B}_0).
\end{array}
\end{equation}
Next, we introduce a proposition on the existence, uniqueness and uniform $L^2$-boundedness of the electrical field $\bvec{E}_\sigma$. This implies that forward operator $F$ is well-defined.
\begin{proposition}\label{prop:E-bd}
Let $\sigma$ satisfy Assumption~\ref{asmp:sigma}. Then the system \eref{eqn:MAT-MI} is uniquely solvable and there exists a constant $C_1$ depending on $\lambda$, $\Lambda$ and $\Omega$, such that
  \[
 \|\bvec{E}_\sigma \|_{L^2} \leq C_1.
 \]
\end{proposition}
\begin{proof}
This proposition will be derived as a consequence of Proposition~\ref{prop:reg}. Let us first reduce the system \eref{eqn:MAT-MI} to a Neumann boundary problem.
  Let $  \bvec{\tilde{E}} = \frac{1}{2}(-y,x,0)$.
   We can readily check that $\nabla \times \bvec{\tilde{E}} = \bvec{B}_1$. Hence $\nabla \times(\bvec{E}_\sigma - \bvec{\tilde{E}} ) = 0$ and we can write $\bvec{E}_\sigma = \bvec{\tilde{E}} + \nabla u$. Substituting this into \eref{eqn:MAT-MI}, we have that $u$ solves the
   Neumann boundary problem,
  \begin{equation}\label{eqn:E-decomp}
    \left\{
     \begin{array}{lll}
   \nabla \cdot (\sigma \nabla u) & = - \nabla \cdot (\sigma \bvec{\tilde{E}}), &\qquad \mbox{ in } \Omega,
  \\
    (\sigma \nabla u + \sigma \bvec{\tilde{E}})\cdot \nu & = 0, &\qquad \mbox{ on } \partial\Omega.
  \end{array}
    \right.
  \end{equation}
  The existence of $u$ and uniqueness of $\nabla u$ follows from Proposition~\ref{prop:reg}. For the uniqueness of $\bvec{E}_\sigma$, we consider the equations
   \begin{equation}\label{homo-eq}
    \left\{
     \begin{array}{lll}
   \nabla \cdot (\sigma \nabla v) & = 0, &\qquad \mbox{ in } \Omega,
  \\
    \sigma \nabla v \cdot \nu & = 0 , &\qquad \mbox{ on } \partial\Omega.
  \end{array}
    \right.
  \end{equation}
  If both $\bvec{E}_1$ and $\bvec{E}_2$ are solutions to the system \eref{eqn:MAT-MI}, then we have that $\bvec{E}_1-\bvec{E}_2 = \nabla v$ and $v$ solves the equations \eref{homo-eq}. By Proposition~\ref{prop:reg}, the only $H^1$ solutions to \eref{homo-eq} are constants. Hence $\nabla v$ vanishes and $\bvec{E}_\sigma$ is unique.

What remains is to show the $L^2$ boundedness of $\bvec{E}_\sigma$.
Applying Proposition~\ref{prop:reg} to $u$, we have that
 \[
 \|\nabla u\|_{L^2} \leq \lambda^{-1} \|\sigma \bvec{\tilde{E}} \|_{L^2}.
 \]
 Hence,
 \[
 \|\bvec{E}_\sigma \|_{L^2} = \|\bvec{\tilde{E}} + \nabla u\|_{L^2} \leq (\Lambda/\lambda + 1) \| \bvec{\tilde{E}} \|_{L^2}.
 \]
 Note that we can choose $\bvec{\tilde{E}} = \frac{1}{2}(-y+a,x+b,0)$ and repeat the above argument for any constants $a$ and $b$. It follows that,
  \[
 \|\bvec{E}_\sigma \|_{L^2} \leq C_1,
 \]
 where
 \[
 C_1 = \frac{1}{2}(\Lambda/\lambda  + 1) \inf_{a,b} \|(-y+a,x+b,0)\|_{L^2},
 \]
 only depends on $\lambda , \Lambda$ and $\Omega$. \cvd

  \end{proof}

\subsection{Linearizaton of the forward map}
Recall that the distribution of the electric field $\bvec{E}_\sigma$
depends nonlinearly on the conductivity $\sigma$ and $\nabla \cdot
(\sigma \bvec{E}_\sigma \times \bvec{B}_0)$ is nonlinear with respect
to $\sigma$. It is natural to start by linearizing the relationship between
conductivity and data. In this
section, we introduce the linearized of the inverse problem. We
first examine the Fr\'echet differentiability of the forward operator
$F$. Then, some useful properties of the Fr\'echet derivative at
$\sigma$, $DF_\sigma$, are presented.

To introduce the Fr\'echet derivative, we consider the following Neumann boundary problem,
\begin{equation}\label{eqn:DF-2}
\left\{
  \begin{array}{lll}
   \nabla \cdot (\sigma \nabla \varphi_h) & = - \nabla \cdot (h \bvec{E}_\sigma), &\qquad \mbox{ in } \Omega,
  \\
     (\sigma\nabla \varphi_h + h\bvec{E}_\sigma) \cdot \nu & =  0 , &\qquad \mbox{ on } \partial\Omega,
  \end{array}
  \right.
\end{equation}
where $h\in W^{1,\infty}$ is the increment to the conductivity.
\begin{theorem}\label{thm:Fre-diff}
For $\sigma$ satisfying Assumption~\ref{asmp:sigma}, the forward operator $F$, defined in \eref{forward-F}, is bounded and Fr\'echet differentiable at $\sigma$. Its Fr\'echet derivative at $\sigma$, $DF_\sigma$, is given by
  \begin{equation}\label{Fre-der}
    DF_\sigma(h) = \nabla \cdot ((\sigma \nabla \varphi_h + h \bvec{E}_\sigma) \times \bvec{B}_0),
  \end{equation}
  where $\varphi_h$ solves \eref{eqn:DF-2}, and satisfies
  \begin{equation}\label{DF-bd}
    \|DF_{\sigma}(h)\|_{L^2} \leq C_2 \|h\|_{W^{1,\infty}}, \quad \forall h\in W^{1,\infty},
  \end{equation}
  for some constant $C_2$ depends on $\lambda, \Lambda$ and $\Omega$.
\end{theorem}
\begin{proof}
We first prove the boundedness of $F$.
We can write
\begin{eqnarray*}
F(\sigma) = \nabla \cdot (\sigma \bvec{E}_\sigma \times \bvec{B}_0)
\\
= \sigma \nabla \cdot ( \bvec{E}_\sigma \times \bvec{B}_0) + \nabla \sigma \cdot ( \bvec{E}_\sigma \times \bvec{B}_0) = \sigma + \nabla \sigma \cdot ( \bvec{E}_\sigma \times \bvec{B}_0).
\end{eqnarray*}
It follows, by boundedness of $\sigma$ and Proposition~\ref{prop:E-bd}, that
\[
\|F(\sigma)\|_{L^2} \leq \|\sigma\|_{L^2} + \|\nabla \sigma\|_{L^\infty} \|\bvec{E}_\sigma\|_{L^2} \leq (|\Omega|^{1/2} + C_1) \|\sigma\|_{W{1,^\infty}},
\]
where $C_1$ is the same constant as in Proposition~\ref{prop:E-bd}.

Next, we show the Fr\'echet differentiability of $F$ at $\sigma$. Consider the data
  \[
  F(\sigma + h)= \nabla \cdot ((\sigma + h) \bvec{E}_{\sigma + h} \times \bvec{B}_0)
   \]
   for some $h \in W^{1,\infty}$ such that $\sigma+h$ also satisfies Assumption~\ref{asmp:sigma}, where $\bvec{E}_{\sigma+h}$ is the solution to \eref{eqn:MAT-MI} with $\sigma$ replaced by $\sigma+h$. Note that
   \[
   \nabla \times (\bvec{E}_{\sigma+h} - \bvec{E}_\sigma) =0.
    \]
  Hence we can write $\bvec{E}_{\sigma+h} - \bvec{E}_\sigma = \nabla u$. Substituting this into the equations for $\bvec{E}_{\sigma+h}$  and $\bvec{E}_\sigma$, we obtain that $u$ solves
  \begin{equation}
  \left\{
  \begin{array}{lll}
   \nabla \cdot (\sigma \nabla u) & = - \nabla \cdot (h \bvec{E}_{\sigma+h}), &\qquad \mbox{ in } \Omega,
  \\
     (\sigma \nabla u + h \bvec{E}_{\sigma+h})\cdot \nu & =  0 , &\qquad \mbox{ on } \partial\Omega.
  \end{array}
  \right.
\end{equation}
 Applying Proposition~\ref{prop:reg} to $u$, we have
 \begin{equation}\label{est-u}
   \|\nabla u\|_{L^2} \leq \lambda^{-1} \|h \bvec{E}_{\sigma+h}\|_{L^2}.
 \end{equation}
 Let $v = u -\varphi_h$, where $\varphi_h$ solves \eref{eqn:DF-2}. Then, $v$ solves
  \begin{equation}
  \left\{
  \begin{array}{lll}
   \nabla \cdot (\sigma \nabla v) & = - \nabla \cdot (h \nabla u), &\qquad \mbox{ in } \Omega,
  \\
    (\sigma \nabla v + h \nabla u) \cdot \nu & =  0 , &\qquad \mbox{ on } \partial\Omega.
  \end{array}
  \right.
\end{equation}
 Applying Proposition~\ref{prop:reg} to $v$, we have
  \begin{equation}\label{est-v}
   \|\nabla v\|_{L^2} \leq \lambda^{-1} \|h \nabla u\|_{L^2}.
 \end{equation}

 To estimate the remainder terms, we write
\begin{eqnarray*}
& F(\sigma +h) - F(\sigma) - \nabla \cdot ((\sigma \nabla \varphi_h + h \bvec{E}_\sigma) \times \bvec{B}_0)
\\
= & \nabla \cdot ((\sigma(\bvec{E}_{\sigma+h} - \bvec{E}_\sigma - \nabla \varphi_h) + h (\bvec{E}_{\sigma+h} -\bvec{E}_\sigma)) \times \bvec{B}_0 )
\\
= & \nabla \cdot ((\sigma \nabla v + h \nabla u) \times \bvec{B}_0)
\\
= & \nabla \sigma \cdot (\nabla v \times \bvec{B}_0)+ \nabla h  \cdot (\nabla u\times \bvec{B}_0).
\end{eqnarray*}
Therefore, by \eref{est-u}, \eref{est-v} and Proposition~\ref{prop:E-bd}, we have
\begin{eqnarray*}
&\| F(\sigma +h) - F(\sigma) - \nabla \cdot ((\sigma \nabla \varphi_h + h \bvec{E}_\sigma) \times \bvec{B}_0) \|_{L^2}
\\
 =  & \|\nabla \sigma \cdot (\nabla v \times \bvec{B}_0)+ \nabla h \cdot (\nabla u\times \bvec{B}_0)\|_{L^2}
 \\
 \leq & \|\nabla \sigma\|_{L^\infty} \|\nabla v \|_{L^2} + \|\nabla h \|_{L^\infty} \|\nabla u\|_{L^2}
 \\
 \leq & C_1 \Lambda \lambda^{-2} \|h\|^2_{L^\infty} + C_1 \lambda^{-1} \|\nabla h \|_{L^\infty} \|h\|_{L^\infty}.
\end{eqnarray*}
We can readily check the linearity of the operator maps $h$ to $ \nabla \cdot ((\sigma \nabla \varphi_h + h \bvec{E}_\sigma) \times \bvec{B}_0)$.
This complete the proof of Fr\'echet differentiability of $F$  at $\sigma$.

What remains is to show that the formal Fr\'echet derivative $DF_\sigma$ is a bounded linear operator. Note that
\begin{eqnarray*}
   DF_\sigma(h) = \nabla \cdot ((\sigma \nabla \varphi_h + h \bvec{E}_\sigma) \times \bvec{B}_0)
\\
= \sigma \nabla \cdot ( \nabla \varphi_h  \times \bvec{B}_0) + \nabla \sigma \cdot ( \nabla \varphi_h  \times \bvec{B}_0)  + h \nabla \cdot ( \bvec{E}_\sigma \times \bvec{B}_0) + \nabla h \cdot ( \bvec{E}_\sigma \times \bvec{B}_0)
 \\
 = h + \nabla \sigma \cdot ( \nabla \varphi_h  \times \bvec{B}_0) + \nabla h \cdot ( \bvec{E}_\sigma \times \bvec{B}_0).
\end{eqnarray*}
By applying Proposition~\ref{prop:reg} to $\varphi_h$ and Proposition~\ref{prop:E-bd} to $\bvec{E}_\sigma$, we conclude that
\begin{eqnarray*}
& \|DF_\sigma(h)\|_{L^2}
\\
=& \|h + \nabla \sigma \cdot ( \nabla \varphi_h  \times \bvec{B}_0) + \nabla h \cdot ( \bvec{E}_\sigma \times \bvec{B}_0)\|_{L^2}
\\
\leq & \|h\|_{L^2} + \Lambda\|\nabla \varphi_h\|_{L^2} + \|\nabla h\|_{L^\infty} \|\bvec{E}_\sigma\|_{L^2}
\\
\leq & |\Omega|^{1/2}\|h\|_{L^\infty} + \Lambda \lambda^{-1} \|h\bvec{E}_\sigma\|_{L^2} + C_1\|\nabla h\|_{L^\infty}
\\
\leq & (|\Omega|^{1/2} + C_1( \Lambda \lambda^{-1} + 1)) \|h\|_{W^{1,\infty}}.
\end{eqnarray*}
\cvd
\end{proof}
%

\section{Uniqueness and stability}
\label{sec:stab}
In the following theorem, we obtain a Lipschitz type stability estimate for the inverse problem under certain conditions on the conductivity. The uniqueness of the  inverse problem follows.

\begin{theorem}\label{thm:linear-stab}
 Suppose that $\sigma$ satisfy Assumption~\ref{asmp:sigma}. If $\sigma$ only depends on the third component of the spatial variable, i.e, $\sigma(x) = \sigma(x_3)$, then the inequality
  \begin{equation}\label{lower-bd-Fre}
    \|DF_{\sigma}(h)\|_{L^2(\Omega} \geq \frac{1}{2} \|h\|_{L^2(\Omega)}
  \end{equation}
  holds true for any $h\in W_0^{1,\infty}(\Omega)$.
\end{theorem}
\begin{proof}
  Note that
  \[
  \nabla \cdot (\nabla \varphi_h \times \bvec{B}_0)=0,
   \]
   for any $H^2$ function $\varphi_h$ and that
   \[
   \nabla \sigma \times \bvec{B}_0 = (0,0, \frac{\partial \sigma}{\partial x_3}) \times (0,0,1) = 0.
   \]
   Hence,
  \[
  \nabla \cdot (\sigma \nabla \varphi_h  \times \bvec{B}_0)
  = \sigma \nabla \cdot (\nabla \varphi_h  \times \bvec{B}_0)
  + \nabla \sigma \cdot ( \nabla \varphi_h  \times \bvec{B}_0)=0.
  \]
    Therefore,
  \[
     DF_{\sigma}(h) = \nabla \cdot ((\sigma \nabla \varphi_h + h \bvec{E}_{\sigma}) \times \bvec{B}_0)
      = \nabla \cdot (h \bvec{E}_{\sigma} \times \bvec{B}_0).
  \]
  Multiplying the both sides by $h$ and integrating over $\Omega$, we obtain that
  \[
    \int_{\Omega} h DF_{\sigma}(h) \, \rmd x  = \int_{\Omega} h \nabla \cdot (h \bvec{E}_{\sigma} \times \bvec{B}_0) \, \rmd x.
  \]
  By using the integration by parts twice, we have
  \begin{eqnarray*}
    \int_{\Omega} h DF_{\sigma}(h) \, \rmd x & = \int_{\Omega} h \nabla \cdot (h \bvec{E}_{\sigma} \times \bvec{B}_0) \, \rmd x
    \\
    & = - \int_{\Omega}  (h \bvec{E}_{\sigma} \times \bvec{B}_0) \cdot \nabla h \, \rmd x
    \\
    & = -\frac{1}{2} \int_{\Omega}  (\bvec{E}_{\sigma} \times \bvec{B}_0) \cdot \nabla (h^2) \, \rmd x
    \\
    & = \frac{1}{2}\int_{\Omega} h^2 \nabla \cdot ( \bvec{E}_{\sigma} \times \bvec{B}_0) \, \rmd x
    \\
    & = \frac{1}{2}\|h\|_{L^2(\Omega)}^2.
  \end{eqnarray*}
  The last identity above follows by noting
  \[
  \nabla \cdot ( \bvec{E}_{\sigma} \times \bvec{B}_0) = \nabla \times \bvec{E}_{\sigma} \cdot \bvec{B}_0 = \bvec{B}_1 \cdot \bvec{B}_0 = 1.
  \]
  Then, by applying Cauchy-Schwarz inequality to $\int_{\Omega} h DF_{\sigma}(h) \, \rmd x$, we obtain \eref{lower-bd-Fre}.  \cvd
\end{proof}

The same technique can be used to provide a general stability estimate for the nonlinear inverse problem.
Note that, in the following theorem and corollary, no smallness constraint on the difference of conductivities is needed.

\begin{theorem}\label{thm:nonlinear-stab}
Suppose that $\sigma_1$ and $\sigma_2$ satisfy Assumption~\ref{asmp:sigma}. If $\sigma_1 -\sigma_2 \in W^{1,\infty}_0(\Omega)$ and
\begin{equation}\label{sig-con}
\nabla \sigma_1 \times \nabla \sigma_2 \cdot \bvec{B}_0 = 0 ,
\end{equation}
then the inequality
  \begin{equation}\label{eqn:nonlinear-stab}
    \|\sigma_1 - \sigma_2\|_{L^2(\Omega)} \leq 2 \|F(\sigma_1) - F(\sigma_2)\|_{L^2(\Omega)},
  \end{equation}
  holds true.
\end{theorem}

\begin{proof}
Assume that $E_1$ and $E_2$ solve \eref{eqn:MAT-MI} with $\sigma$ replaced by $\sigma_1$ and $\sigma_2$, respectively.
Let us multiply $F(\sigma_1) - F(\sigma_2)$ by $\sigma_1 - \sigma_2$ and integrate over $\Omega$ to obtain
\begin{equation}\label{iden-1}
\eqalign{
  & \int_\Omega (\sigma_1 - \sigma_2) (F(\sigma_1) - F(\sigma_2)) \, \rmd x
  \\
  = & \int_{\Omega}(\sigma_1 - \sigma_2) (\nabla \cdot(\sigma_1 \bvec{E}_1 - \sigma_2 \bvec{E}_2) \times \bvec{B}_0)  \, \rmd x
  \\
  = & \int_{\Omega}(\sigma_1 - \sigma_2) (\nabla \cdot((\sigma_1 -\sigma_2) \bvec{E}_1 \times \bvec{B}_0) + \nabla \cdot(\sigma_2 (\bvec{E}_1 - \bvec{E}_2)\times \bvec{B}_0))  \, \rmd x
  \\
  = & \frac{1}{2} \|\sigma_1 - \sigma_2\|_{L^2(\Omega)}^2
 +\int_{\Omega}(\sigma_1 - \sigma_2)  \nabla \cdot(\sigma_2 (\bvec{E}_1 - \bvec{E}_2)\times \bvec{B}_0)  \, \rmd x.
}
\end{equation}
In the above inequalities, the last step follows by the similar argument as in the proof of Theorem~\ref{thm:linear-stab}.

Next, we estimate
\[
\int_{\Omega}(\sigma_1 - \sigma_2)  \nabla \cdot(\sigma_2 (\bvec{E}_1 - \bvec{E}_2)\times \bvec{B}_0)  \, \rmd x.
\]
Recall that $\nabla \times (\bvec{E}_1 - \bvec{E}_2) = 0$. Hence, we can write $\bvec{E}_1 - \bvec{E}_2 = \nabla u$. Applying integration by parts twice, we obtain that
\begin{equation}\label{iden-2}
\eqalign{
& \int_{\Omega}(\sigma_1 - \sigma_2)  \nabla \cdot(\sigma_2 (\bvec{E}_1 - \bvec{E}_2)\times \bvec{B}_0)  \, \rmd x
\\
= & -\int_{\Omega} \sigma_2 \nabla(\sigma_1 - \sigma_2) \cdot (\nabla u \times \bvec{B}_0)  \, \rmd x
\\
= & \int_{\Omega} \sigma_2 (\nabla(\sigma_1 - \sigma_2)\times \bvec{B}_0) \cdot \nabla u \, \rmd x
\\
= & -\int_{\Omega} \nabla \cdot (\sigma_2 (\nabla(\sigma_1 - \sigma_2)\times \bvec{B}_0)) u \, \rmd x
\\
= & -\int_{\Omega} (\sigma_2 \nabla \cdot (\nabla(\sigma_1 - \sigma_2)\times \bvec{B}_0) + \nabla \sigma_2 \times \nabla(\sigma_1 - \sigma_2) \cdot\bvec{B}_0) u \, \rmd x
\\
= & 0.
}
\end{equation}
Here we use the equalities \eref{sig-con},
\[
\nabla \cdot (\nabla(\sigma_1 - \sigma_2)\times \bvec{B}_0) = 0 ,
\]
and
\[
\nabla \sigma_2 \times \nabla \sigma_2 = 0.
\]

Combining \eref{iden-1} and \eref{iden-2}, we discover
\[
\int_\Omega (\sigma_1 - \sigma_2) (F(\sigma_1) - F(\sigma_2)) \, \rmd x = \frac{1}{2} \|\sigma_1 - \sigma_2\|_{L^2(\Omega)}^2.
\]
The stability estimate \eref{eqn:nonlinear-stab} follows by applying the Cauchy-Schwarz inequality to the left-hand side of the above equality. \cvd
\end{proof}

In the following corollary, we list some simple cases, in which, the criteria \eref{sig-con} is easy to check.

\begin{corollary}\label{coro:stab-cont}
 Suppose that $\sigma_1$ and $\sigma_2$ satisfy Assumption~\ref{asmp:sigma}. If $\sigma_1 -\sigma_2 \in W^{1,\infty}_0(\Omega)$ and satisfy any one of the following three conditions:
 \begin{enumerate}
   \item $\sigma_1$ only depends on the third component of the spatial variable $x_3$;
   \item  There exists a real number $t$ such that $t \sigma_1 + (1-t) \sigma_2$ only depends on $x_3$;
   \item There exist a positive integer $N$ and real numbers $a_n$, $n=1,2\dots N$ such that $\sigma_2 + \sum_{n=1}^N a_n(\sigma_1 -\sigma_2)^n$ only depends on $x_3$;
 \end{enumerate}
then the stability estimate \eref{eqn:nonlinear-stab} holds true.
\end{corollary}
\begin{proof}
  We can readily see that $(i)$ and $(ii)$ are simple cases of $(iii)$. It suffices to show that \eref{sig-con} is satisfied and apply Theorem~\ref{thm:nonlinear-stab}.

  From $(iii)$, we know that
  \begin{equation}\label{help1}
  \nabla \left(\sigma_2 + \sum_{n=1}^N a_n(\sigma_1 -\sigma_2)^n \right) \times \bvec{B}_0 = \bvec{0}.
  \end{equation}
  In light of \eref{help1} and the facts that
  \[
  \nabla \sigma_2 \times  \nabla \sigma_2 = \bvec{0},
  \]
  we have the following equalities,
  \begin{eqnarray*}
      & \nabla \sigma_1 \times \nabla \sigma_2 \cdot \bvec{B}_0
    \\
    = & \nabla (\sigma_1 -\sigma_2) \times \nabla \sigma_2 \cdot \bvec{B}_0
    \\
    = & \nabla (\sigma_1 -\sigma_2) \times \nabla \left(\sigma_2 - \left(\sigma_2 + \sum_{n=1}^N a_n(\sigma_1 -\sigma_2)^n \right)\right) \cdot \bvec{B}_0
   \\
    = & \nabla (\sigma_1 -\sigma_2) \times \left(\sum_{n=1}^N a_n n(\sigma_1 -\sigma_2)^{n-1} \nabla (\sigma_1 -\sigma_2) \right)\cdot \bvec{B}_0
    \\
    = &  \left(\sum_{n=1}^N a_n n(\sigma_1 -\sigma_2)^{n-1} \nabla (\sigma_1 -\sigma_2) \times \nabla (\sigma_1 -\sigma_2) \right)\cdot \bvec{B}_0
    \\
    = & 0.
  \end{eqnarray*}
  The proof is completed by applying Theorem~\ref{thm:nonlinear-stab}.  \cvd
\end{proof}

Roughly speaking, in Theorem~\ref{thm:nonlinear-stab}, we prove that, if the structure of two conductivities satisfies the condition \eref{sig-con}, the inverse problem bears a Lipschitz stability estimate. We propose next to remove this structure condition.
In Theorem~\ref{thm:nonlinear-stab2}, we show that, if one conductivity varies less dramatically, the Lipschitz type stability estimates also holds true.

\begin{proof}[Proof of Theorem~\ref{thm:nonlinear-stab2}]
  The proof differs from the one of Theorem~\ref{thm:nonlinear-stab} in the treatment of the last term in \eref{iden-1},
  \[
  I \triangleq \int_{\Omega}(\sigma_1 - \sigma_2)  \nabla \cdot(\sigma_2 (\bvec{E}_1 - \bvec{E}_2)\times \bvec{B}_0)  \, \rmd x.
  \]
  We continue from \eref{iden-1}. First, we estimate the electric field difference. Note that $\bvec{E}_1 - \bvec{E}_2$ is curl-free and we set
  \[
  \nabla u = \bvec{E}_1 - \bvec{E}_2.
  \]
  Then, $u$ satisfies the equation
  \begin{equation}
  \left\{
  \begin{array}{lll}
   \nabla \cdot (\sigma_1 \nabla u) & =  -\nabla \cdot ((\sigma_1 - \sigma_2)\bvec{E}_2), &\qquad \mbox{ in } \Omega,
  \\
     \nabla u \cdot \nu & =  0 , &\qquad \mbox{ on } \partial\Omega.
  \end{array}
  \right.
\end{equation}
  Applying Proposition~\ref{prop:reg} to $u$, we obtain that
  \[
  \|\nabla u\|_{L^2} \leq \lambda^{-1} \|(\sigma_1 -\sigma_2) \bvec{E}_2\|_{L^2}.
  \]
  From the standard $L^p$ estimate of elliptic equations \cite[Chapter 9]{Gilbarg2001} and the Sobolev Embedding Theorem, we know that $\bvec{E}_2$ is bounded and
  \[
  \|\bvec{E}_2\|_{L^\infty} < C,
  \]
  where $C$ only depends on $r_0$, $\lambda$, $\Lambda$ and $\Omega$. Thus, we conclude that
  \[
  \|\bvec{E}_1 - \bvec{E}_2\|_{L^2} \leq C \|(\sigma_1 -\sigma_2)\|_{L^2}.
  \]
  Now, with the choice of $K$ such that $KC \leq 1/4$,
  we estimate $|I|$ as follows:
  \begin{equation}\label{I-est2}
\eqalign{
  |I|  & = \left| \int_{\Omega}(\sigma_1 - \sigma_2)  \nabla \cdot(\sigma_2 (\bvec{E}_1 - \bvec{E}_2)\times \bvec{B}_0)  \, \rmd x \right|
  \\
   & =\left|\int_{\Omega}(\sigma_1 - \sigma_2)  \nabla \sigma_2 \cdot( (\bvec{E}_1 - \bvec{E}_2)\times \bvec{B}_0)  \, \rmd x \right|
  \\
  & \leq  \|\nabla \sigma_2\|_{L^\infty} \, \|\sigma_1 -\sigma_2\|_{L^2} \, \|\bvec{E}_1 - \bvec{E}_2\|_{L^2}
  \\
  & \leq  C\|\nabla \sigma_2\|_{L^\infty} \, \|\sigma_1 -\sigma_2\|_{L^2}^2
  \\
  & \leq \frac{1}{4} \|\sigma_1 -\sigma_2\|_{L^2}^2.
}
\end{equation}
Substituting \eref{I-est2} into \eref{iden-1}, we discover that
\[
\int_\Omega (\sigma_1 - \sigma_2) (F(\sigma_1) - F(\sigma_2)) \, \rmd x \geq \frac{1}{4} \|\sigma_1 - \sigma_2\|_{L^2(\Omega)}^2.
\]
The desired estimate \eref{eqn:nonlinear-stab2} follows by applying the Cauchy-Schwarz inequality to the left-hand side. \cvd
\end{proof}

\section{An iterative reconstruction scheme}
One possible approach to solving the inverse problem is to formulate
it as a least-squares problem.  One can then apply a gradient-based
method to solve the least-squares problem.  Such a method will require
knowledge of the Fr\'echet derivative of the forward map which
we studied in Section~\ref{sec:model}.  Convergence analysis of
this type of reconstruction approach is available in
\cite{Hoop2012,Hoop2015}.  Results in these references, together with
our analysis of of $DF$ in Sections~\ref{sec:model} and
\ref{sec:stab} can be used to provide a convergence
analysis for the iterative reconstruction of MAT-MI using steepest
descent method. The main challenge of the least-squares approach
lies in the difficulty to accurately evaluate $DF$ and its adjoint
where numerical differentiations are involved.  We temporarily
abandon the least-squares approach in favor of one
that is based on a fixed point method.  This approach is described next.

\subsection{Formulation}
In view of the structure of this inverse problem, we propose a novel
iterative scheme, in which, the forward map and its derivative are not
required. The desired conductivity is updated by solving
a stationary advection-diffusion equation. Let
$\sigma^\dagger$ denote the unknown conductivity to be reconstructed,
$\bvec{E}^\dagger$ be the corresponding electric field and $g$ be the
internal data obtained in the first step of MAT-MI. The internal data
is related to the conductivity and the field through
\[
g = \nabla \cdot (\sigma^\dagger \bvec{E}^\dagger \times \bvec{B}_0).
\]
The algorithm proceeds as follows:
\begin{enumerate}[(S1)]
\setcounter{enumi}{-1}
  \item Select an initial conductivity $\sigma_0$ and set $k=0$;
  \item Calculate the associated electric field $\bvec{E}_k$ by solving the
        boundary value problem
      \begin{equation}
      \left\{
      \begin{array}{lll}
        \nabla \times \bvec{E}_k & = \bvec{B}_1, &\qquad \mbox{ in }
        \Omega,
          \\
        \nabla \cdot (\sigma_k \bvec{E}_k) & =0, &\qquad \mbox{ in } \Omega,
          \\
          \sigma_k \bvec{E}_k \cdot \nu &= 0 , &\qquad \mbox{ on }
            \partial\Omega .
       \end{array}
       \right.
      \end{equation}
  \item Calculate the updated conductivity by solving the stationary advection-diffusion equation:
      \begin{equation}\label{eqn:update_sigma}
      \left\{
      \begin{array}{rll}
        \nabla \cdot (\sigma_{k+1} \bvec{E}_{k}\times \bvec{B}_0) & = g , &\qquad \mbox{ in } \Omega,
          \\
        \sigma_{k+1} & = \sigma_0 , & \qquad \mbox{ on } \partial\Omega.
        \end{array}
       \right.
      \end{equation}
  \item Set $k=k+1$ and go to (S1).
 \end{enumerate}
Convergence test can be based on $\|\sigma_{k}-\sigma_{k-1}\|$ or based on data misfit
$\| g - \nabla \cdot (\sigma_{k+1} \bvec{E}_{k+1} \times \bvec{B}_0) \|$.

\subsection{Convergence analysis}
The main advantage of this scheme is two-fold:
First, the update of the conductivity is calculated directly using
the the measured data and the simulated electric field. Hence, fewer
numerical differentiations are involved when compared to the gradient-based
least-squares minimization.
Second, the convergence analysis can be carried out
using an idea similar to the one in the proof of
 Theorem~\ref{thm:nonlinear-stab2}. A global convergence result and a
 linear convergence rate are established the following theorem.
  \begin{theorem}\label{thm:num-conv}
    Suppose that the true conductivity $\sigma^\dagger$ satisfies Assumption~\ref{asmp:sigma} and
    \begin{equation}\label{sig-slow-alg}
        \|\nabla \sigma^\dagger \|_{L^\infty} < 2K,
    \end{equation}
     where the constant $K$ is the same as in Theorem~\ref{thm:nonlinear-stab2}, which only depends on $r_0$, $\lambda$, $\Lambda$ and $\Omega$. Then, for any initial $\sigma_0$ satisfying Assumption~\ref{asmp:sigma} and coinciding with $\sigma^\dagger$ over the boundary $\partial \Omega$, the above algorithm generates a sequence $\{ \sigma_k\} , \quad k=0,1,\dots$, which is convergent to $\sigma^\dagger$ and satisfies
     \begin{equation}\label{conv-rate}
       \|\sigma_{k} - \sigma^\dagger\|_{L^2} \leq c^k \|\sigma_0 - \sigma^\dagger\|_{L^2}, \quad k= 0,1,\dots ,
     \end{equation}
     where $c<1$ depends on $\|\nabla \sigma^\dagger \|_{L^\infty}$ and $\Omega$.
  \end{theorem}
\begin{proof}
  We start by subtracting $\nabla\cdot(\sigma^\dagger \bvec{E}_k \times \bvec{B}_0)$ from both sides of \eref{eqn:update_sigma} to obtain
  \[
   \nabla \cdot((\sigma_{k+1} - \sigma^\dagger) \bvec{E}_k \times\bvec{B}_0) = \nabla \cdot(\sigma^\dagger (\bvec{E}^\dagger - \bvec{E}_k) \times\bvec{B}_0).
  \]
  Multiplying the both sides by $\sigma_{k+1} - \sigma^\dagger$ and integrating over $\Omega$, we arrive at
  \begin{eqnarray*}
  & \frac{1}{2} \|\sigma_{k+1} - \sigma^\dagger\|^2_{L^2}
  \\
  = &  \int_{\Omega} \, (\sigma_{k+1} - \sigma^\dagger) \nabla \cdot((\sigma_{k+1} - \sigma^\dagger) \bvec{E}_k \times\bvec{B}_0) \dd x
   \\
  = & \int_{\Omega} \, (\sigma_{k+1} - \sigma^\dagger) \nabla \cdot(\sigma^\dagger (\bvec{E}^\dagger - \bvec{E}_k) \times\bvec{B}_0) \dd x
  \\
  = & \int_{\Omega} \, (\sigma_{k+1} - \sigma^\dagger) \nabla \sigma^\dagger \cdot( (\bvec{E}^\dagger - \bvec{E}_k) \times\bvec{B}_0) \dd x.
  \end{eqnarray*}
  In the above identities, the first identity follows from a similar argument to the one used in the proof of Theorem~\ref{thm:linear-stab} and the last identity follows by noting that $(\bvec{E}^\dagger - \bvec{E}_k) \times\bvec{B}_0$ is divergence-free. Next, we estimate the electric field difference. As in the proof of Theorem~\ref{thm:nonlinear-stab2}, we conclude that
  \[
    \|\bvec{E}^\dagger - \bvec{E}_k\|_{L^2} \leq C \|\sigma_k - \sigma^\dagger\|_{L^2}.
  \]
  By the Cauchy-Schwarz inequality and \eref{sig-slow-alg}, we have that
  \[
  \|\sigma_{k+1} - \sigma^\dagger\|_{L^2} \leq c \|\sigma_k - \sigma^\dagger\|_{L^2}
  \]
  and \eref{conv-rate} follows from an induction argument on $k$. \cvd

\end{proof}

\begin{remark}
  Let us point out that indeed the convergence analysis of the
  proposed algorithm carries through when the inverse problem have a
  Lipschitz type stability estimate. In fact,
  Theorem~\ref{thm:nonlinear-stab2} still holds true with the
  condition \eref{sig-slow} replaced by
  \eref{sig-slow-alg}. Correspondingly, the stability constant will
  depend on $\|\nabla \sigma^\dagger \|_{L^\infty}$ and blow up as
  $\|\nabla \sigma^\dagger \|_{L^\infty}$ approaches $2K$.
\end{remark}

\subsection{Numerical experiments}
Now we present some numerical experiments to verify the convergence
theory presented in the previous subsection. For each experiment, the
true conductivity is assumed to be Lipschitz continuous and equal to
$0.2$ near the boundary and we use constant $0.2$ as the initial model
unless otherwise specified. To simplify the computation, we transform
the 3D problem into a 2D problem by assuming the conductivity is
invariant along the $x_3$ direction.  The setup is as follows. The
domain we take is the square $\Omega = (0, 1) \times (0,1)$. We employ
a uniform triangulation with a mesh size of $1/64$. Both the Neumann
problem and the stationary advection-diffusion equation are solved
using a first-order finite element method. The algorithm is
implemented using FEniCS, a finite element software package
\cite{Logg2012}, and using Python as the user interface. All the
numerical computations are performed on a dual-core laptop computer.

\subsubsection*{Example 1.}
We first consider a simple example. The true conductivity is shown in
\Fref{fig:1peak-true} and the error between the true and reconstructed
model is shown in \Fref{fig:1peak-error}. The relative $L^2$-error, 
$\|\sigma_k-\sigma^\dagger\|/\|\sigma^\dagger\|$,
drops to $2.88\times10^{-7}$ after $16$ iterations. As shown in
\Fref{fig:1peak-decay}, a linear convergence rate is observed.
\begin{figure}
  \centering
  \subfloat[]{
      \includegraphics[width=.45\textwidth,clip=true,trim=5cm 1cm 0.2cm 2cm]{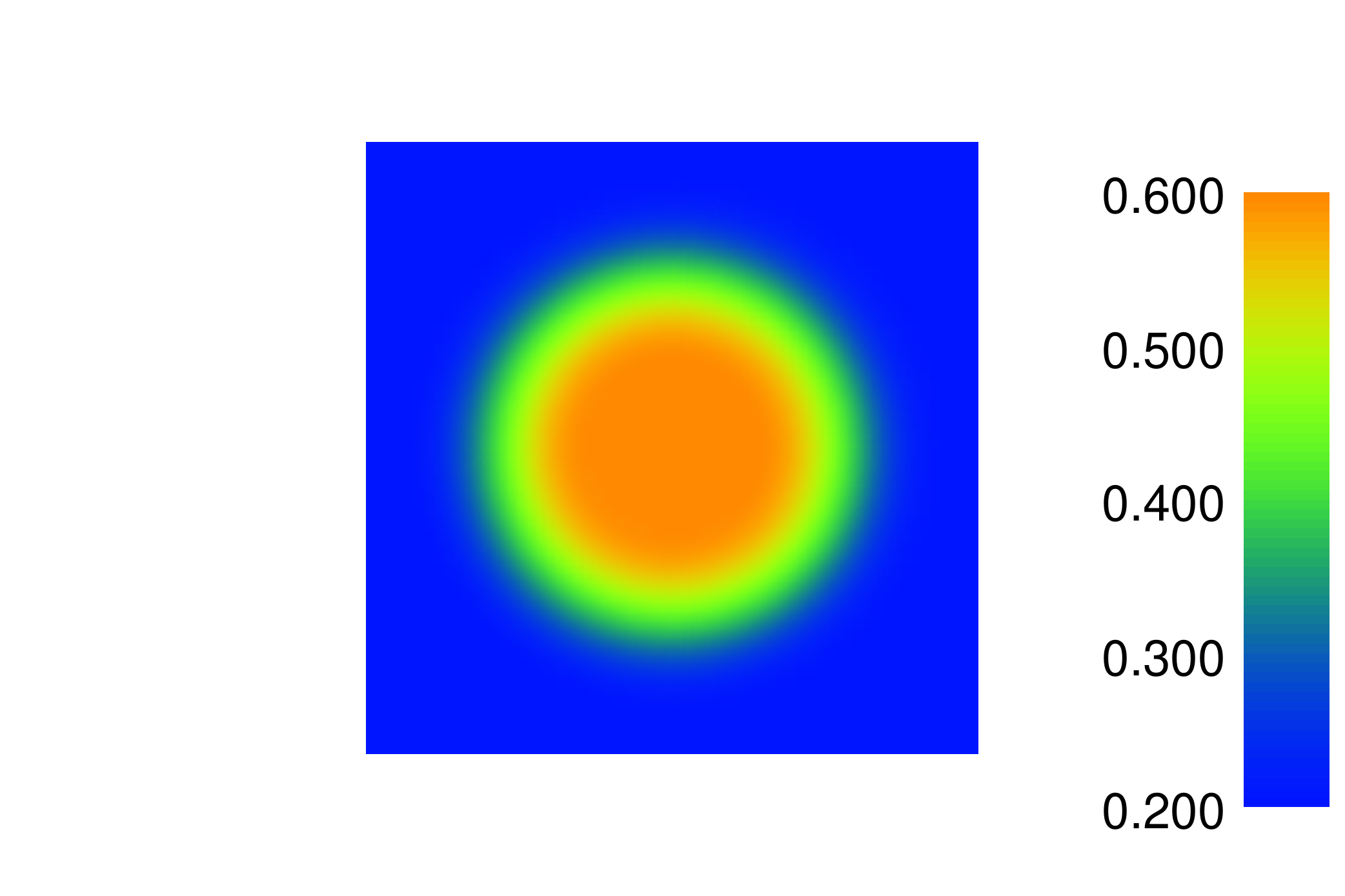}
      \label{fig:1peak-true}}
  \hspace{0.6mm}
  \subfloat[]{
      \includegraphics[width=.45\textwidth,clip=true,trim=5cm 1cm 0.2cm 2cm]{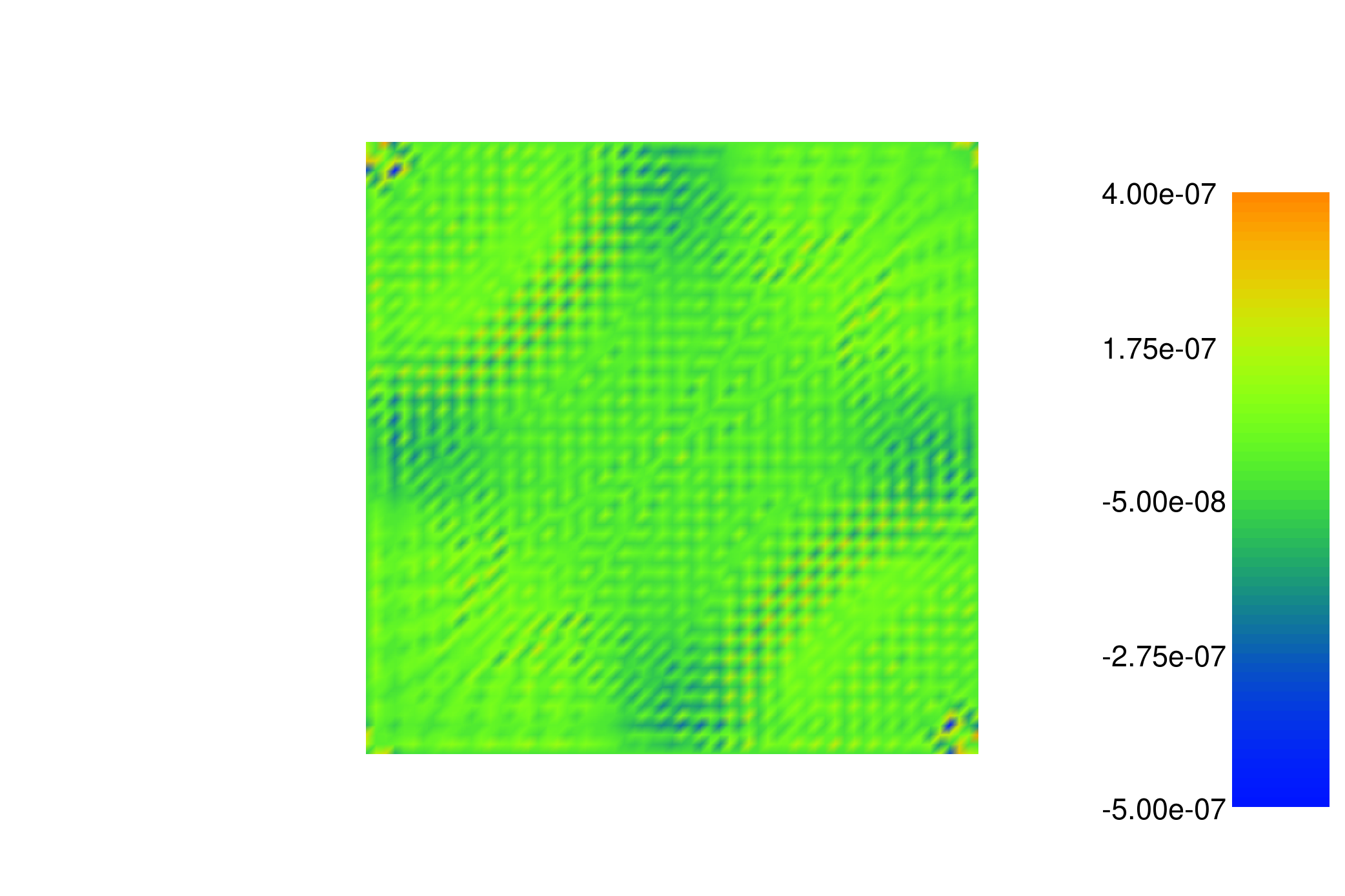}
      \label{fig:1peak-error}  }
  \\
  \subfloat[]{
  \begin{overpic}[width=.75\textwidth,clip=true, trim=0cm -0.6cm 0cm 0cm]{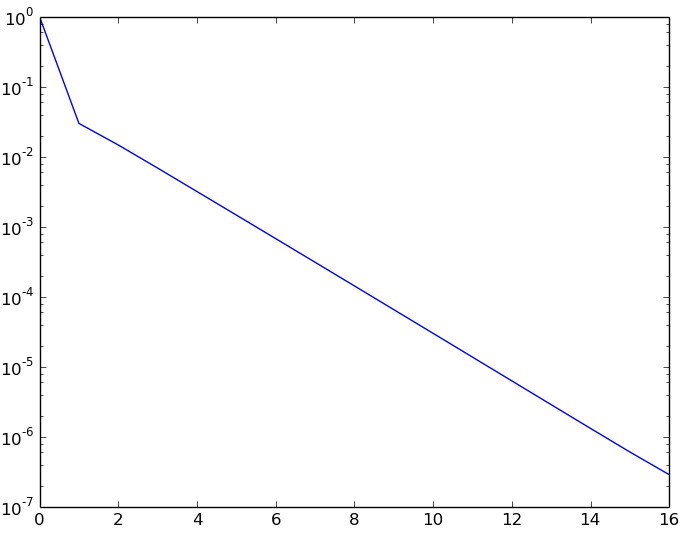}
      \put(-5.5,40){\rotatebox[]{90}{ Relative error ($L^2$) }}
      \put(36,0.5){ Iteration number}
      \end{overpic}
      \label{fig:1peak-decay}  }
  \caption{(a) true conductivity model, (b) absolute error between the inverted and true conductivity model, (c) decay of the relative error in $L^2$ norm (logarithmic scale).}
  \label{fig:ex1}
\end{figure}

\subsubsection*{Example 2.}
We then attempt to recover a more complicated conductivity model, as
shown in \Fref{fig:M-true}. In \Fref{fig:M-error}, we show the
absolute error of the recovered model after $45$ iterations. The
relative $L^2$-error drops to $2.57 \times 10^{-7}$ and a linear
convergence rate is still observed. In this example, the gradient of
the conductivity is greater than the one in the previous
example. According to Theorem~\ref{thm:num-conv}, this will lead to a
greater prefactor $c$ in the convergence rate. The comparison of
\Fref{fig:1peak-decay} and \Fref{fig:M-decay} demonstrates this point.

\begin{figure}
  \centering
  \subfloat[]{
      \includegraphics[width=.45\textwidth,clip=true,trim=5cm 1cm 0.2cm 2cm]{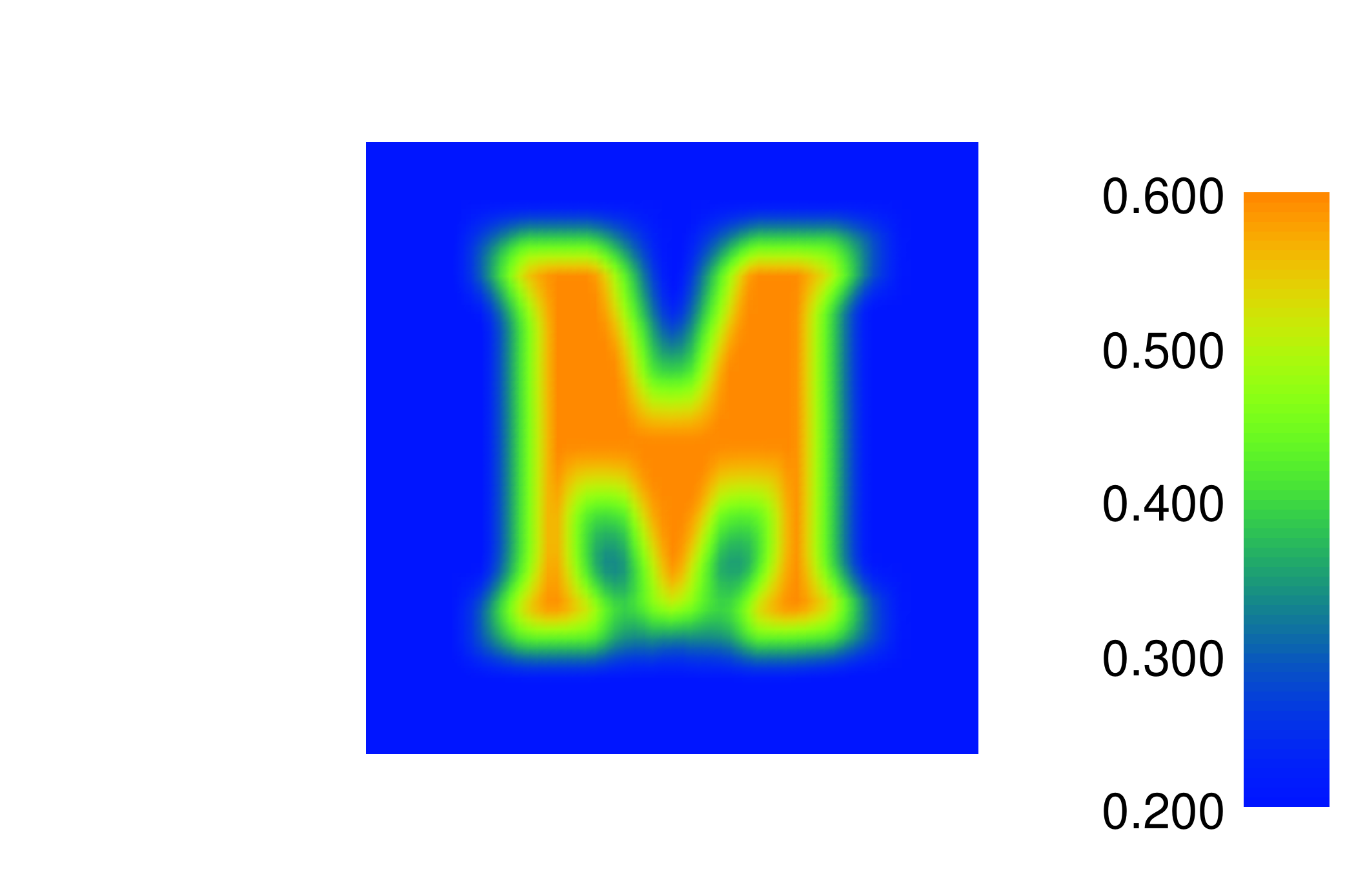}
      \label{fig:M-true}}
  \hspace{0.6mm}
  \subfloat[]{
      \includegraphics[width=.45\textwidth,clip=true,trim=5cm 1cm 0.2cm 2cm]{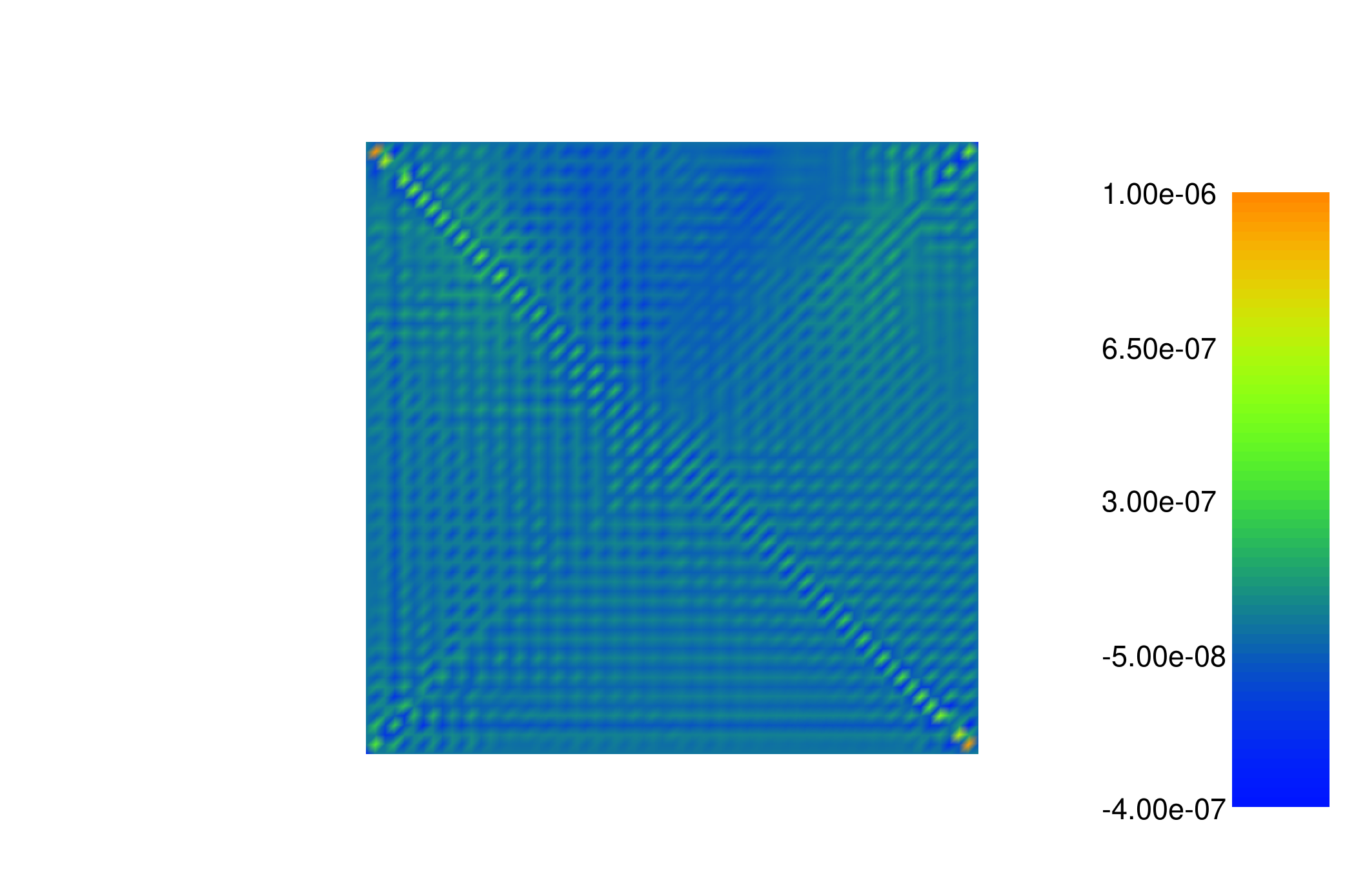}
      \label{fig:M-error}  }
  \\
  \subfloat[]{
      \begin{overpic}[width=.75\textwidth,clip=true, trim=0cm -0.6cm 0cm 0cm]{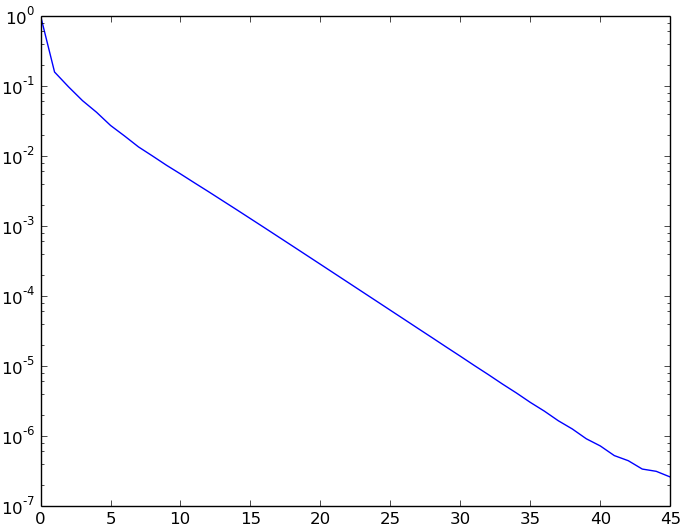}
      \put(-5.5,40){\rotatebox[]{90}{ Relative error ($L^2$) }}
      \put(36,0.5){ Iteration number}
      \end{overpic}
      \label{fig:M-decay}  }
  \caption{(a) true conductivity model, (b) absolute error between the inverted and true conductivity model, (c) decay of the relative error in $L^2$ norm (logarithmic scale).}
  \label{fig:ex2}
\end{figure}

\subsubsection*{Example 3.}
To further investigate and demonstrate the convergence results in Theorem~\ref{thm:num-conv}, we perform the third test, which is the ``reverse'' Example~2. We switch the role of the true model and the initial model in Example~2. That is, we try to recover the constant conductivity with an initial model as shown in \Fref{fig:M-true}. The algorithm converges after $1$ iteration with the absolute $L^2$-error drops below $5\times 10^{-8}$. This implies that the prefactor $c$ approaches zero as the true conductivity goes to a constant function. Actually, this can be proved by noticing that, when $g$ is constant, the unique solution to \eref{eqn:update_sigma} is the same constant for any admissible $\sigma_0$.

\section{Discussion}
\label{sec:dis}
We investigated the second step in MAT-MI where the problem is
to reconstruct the conductivity distribution from internal data
obtained in the first step.  A global Lipschitz type stability estimate
is established when the conductivity is $W^{1,\infty}$.  We devise
a novel iterative method for solving the inverse problem that involves,
at each iteration, the solution of a well-posed boundary value problem
followed by the solution of an advection-diffusion problem.  The
iterative method is shown to be convergent.  Results from numerical
experiments demonstrate the effectiveness of the approach.

It would be interesting to extend the computational method proposed
to three dimensions and to invert real measured data.  An important
direction for this research is to consider the case of anisotropic
conductivity.  In \cite{Brinker2008}, the authors examine the effect of electrical
anisotropy in MAT-MI. A homogeneous tissue is considered. They find
that, when imaging nerve or muscle, electrical anisotropy has a
significant effect on the acoustic signal and must be accounted for in
order to obtain accurate images.

\section*{Acknowledgements}
The authors would like to thank Professor Bin He, Leo Mariappan, and Zhu Wang for their helpful discussions.
This research was supported in part by the Institute for Mathematics and its Applications with funds provided by the National Science Foundation under NSF DMS-0931945. Fadil Santosa's research is supported in
part by NSF DMS-1211884.


\noindent

\def\cprime{$'$} \newcommand{\SortNoop}[1]{}

\section*{References}
\bibliographystyle{siam}
\bibliography{conv}

\end{document}